\documentclass[preprint,11pt]{elsarticle}
\usepackage[latin1]{inputenc}
\usepackage[english]{babel}
\usepackage{geometry}
\usepackage{amsmath}
\usepackage{amssymb}
\usepackage{oldgerm}
\usepackage{amscd}
\usepackage{rotating}
\usepackage[all]{xy}
\usepackage{hyperref}

\usepackage{color}

\date{\today}
\newtheorem{theorem}{Theorem}[section]
\newtheorem{proposition}[theorem]{Proposition}
\newtheorem{lemma}[theorem]{Lemma}
\newtheorem{corollary}[theorem]{Corollary}

\newdefinition{definition}[theorem]{Definition}
\newdefinition{example}[theorem]{Examples}
\newdefinition{remark}[theorem]{Remark}

\newproof{proof}{Proof}

\newcommand{\Q}{\mathbf Q}
\newcommand{\N}{\mathbf N}
\newcommand{\Z}{\mathbf Z}

\newcommand{\p}{\mathfrak p}

\newcommand{\cO}{{\mathcal O}}

\newcommand{\Int}{{\rm Int}}
\newcommand{\Gal}{{\rm Gal}}
\newcommand{\tr}{{\rm tr}}

\include{thebibliography}

\begin{document}

\begin{frontmatter}

\title{Galois structure on integral valued polynomials\\
 \vskip0.3cm
\small{accepted for publication in J. Number Theory (2016). Arxiv:\href{http://arxiv.org/abs/1511.01295}{http://arxiv.org/abs/1511.01295} } }
%Doi\href{}{}

\author{Bahar Heidaryan}
\address{Dipartimento di Matematica, Universit\`a di Padova, Via Trieste 63, 35121 Padova, Italy \& \\
Department of Mathematics, Tarbiat Modares University, 14115-134, Tehran, Iran.}
\ead{b.heidaryan@modares.ac.ir}

\author{Matteo Longo}
%\address{Dipartimento di Matematica, Universit\`a di Padova, Via Trieste 63, 35121 Padova, Italy}
\ead{mlongo@math.unipd.it}

\author{Giulio Peruginelli} 
\ead{gperugin@math.unipd.it}
\address{Dipartimento di Matematica, Universit\`a di Padova, Via Trieste 63, 35121 Padova, Italy}
%\fntext[fn1]{Phone number: +43 (316) 873 - 7633}

\begin{abstract} We characterize finite Galois extensions $K$ of the field of rational numbers in terms of 
the rings $\Int_{\Q}(\mathcal O_K)$, recently introduced by Loper and Werner, consisting
of those polynomials which have coefficients in $\Q$ and 
such that $f(\mathcal O_K)$ is contained in  $\mathcal O_K$. 
We also address the problem of constructing a  basis for $\Int_{\Q}(\mathcal O_K)$ as a $\Z$-module. 
\end{abstract}


\end{frontmatter}

%11R  Algebraic number theory: global fields % 11R32 Galois theory 
%11S Algebraic number theory: local and p-adic fields 
%11S20 Galois theory 

\section{Introduction}
The main object of this paper is to study the class of rings 
$$\Int_\Q(\mathcal O_K):=\Int(\mathcal O_K)\cap \Q[X]$$
where $K$ varies among the set of finite Galois extensions of $\Q$; here $\mathcal O_K$ is the ring of algebraic integers of $K$ 
and $\Int(\mathcal O_K)$ is the ring of polynomials $f\in K[X]$ such that $f(\mathcal O_K)$ is contained in $\mathcal O_K$. 

The rings $\Int_\Q(\mathcal O_K)$ have been introduced in \cite{LopWer} and studied also in \cite{PerIntvalbounded}.
Among other things, the authors of \cite{LopWer} proved that $\Int_\Q(\mathcal O_K)$ is a Pr\"ufer domain. 
It is immediate to see that $\Int_\Q(\mathcal O_K)$ is contained in 
$$\Int(\Z)=\{f\in\Q[X] \mid f(\Z)\subseteq\Z\},$$
the classical ring of integer-valued polynomials. Moreover, if $K$ is a proper field extension of $\Q$, then $\Int_\Q(\mathcal O_K)$ is properly contained in $\Int(\Z)$: in fact, let $p\in\Z$ be a prime which is not totally split in $\mathcal O_K$; then it is not difficult to see that the polynomial 
\[f(X)=\frac{X(X-1)\ldots(X-(p-1))}{p}\]
is in $\Int(\Z)$ but not in $\Int_{\Q}(\mathcal O_K)$. This is an evidence of the fact that, for the class of finite Galois extension $K/\Q$, the ring $\Int_{\Q}(\mathcal O_K)$ is completely determined by the set of primes $p\in\Z$ which are totally split in $\mathcal O_K$, and therefore by the field $K$ itself. Our main result is a characterization of finite Galois extensions of $\Q$ in terms of the rings $\Int_\Q(\mathcal O_K)$. More precisely, as a corollary 
of our main result Theorem  \ref{equalIntQOK}, we prove the following: 

\begin{theorem} Let $K$ and $K'$ be finite Galois extensions of $\Q$. Then 
$\Int_\Q(\mathcal O_K)=\Int_\Q(\mathcal O_{K'})$ if and only if $K=K'$. 
\end{theorem}

The statement is false if we consider finite extensions of $\Q$ which are not Galois. In fact, if $K/\Q$ is a finite non-Galois extension and $K'$ is any conjugate field of $K$ over $\Q$ different from $K$, then it is easy to see that $\Int_\Q(\mathcal O_K)=\Int_\Q(\mathcal O_{K'})$. 

 A study somehow related to the present paper about the so-called polynomial overrings of $\Int(\Z)$, that is, rings $R$ such that $\Int(\Z)\subseteq R\subseteq\Q[X]$, has been recently done in \cite{ChabPer}, in which a complete and thorough classification of such rings has been given: each of them can be realized as the ring of integer-valued polynomials over some closed subset of the profinite completion of $\Z$. 

We can reformulate our main result in more abstract terms as follows. Denote by $\mathcal G$ the category whose objects are 
ring of integers $\mathcal O_K$ of finite Galois extensions $K/\Q$ with homomorphism given by inclusions, 
and by $\mathcal C$ the category of subrings of $\Q[X]$ in which morphisms are again inclusions. 
Then the functor 
\[\Int_\Q: \mathcal G\longrightarrow \mathcal C\] 
which takes an object $\mathcal O_K$ of $\mathcal G$ to $\Int_\Q(\mathcal O_K)$ and 
the inclusion $\mathcal O_K\subseteq \mathcal O_{K'}$ 
to the inclusion $\Int_\Q(\mathcal O_{K'})\subseteq \Int_\Q(\mathcal O_K)$,  
is a  full and  faithful contravariant functor ( see also Remark \ref{alternative proof}). 

We next address the problem of constructing a regular basis of $\Int_{\Q}(\mathcal O_K)$ as a $\Z$-module ( see Section \ref{section2} for the definition of regular basis). 
In particular, we discuss the value of the $p$-adic valuation of the leading term of the element of degree $n$ in a regular basis, for each prime number $p$. We show that this is equivalent to understanding the analogous question  for the ring 
\[\Int_{\Q_p}(K):=\Int(\mathcal O_K)\cap \Q_p[X]\] 
for each finite extension $K/\Q_p$, where $\Int(\mathcal O_K)$ is 
the ring of $f\in K[X]$ such that $f(\mathcal O_K)\subseteq \mathcal O_K$, and $\mathcal O_K$ is the valuation ring of $K$. 
We completely determine these values in Theorem \ref{equality-theorem}, in the case of tame ramification. 
As a consequence, we obtain the second main result of this paper. 
To state the theorem, let $K/\Q$ be a Galois extension and, 
for any prime $p$ of $\Z$, let $q_p$ and $e_p$ be the cardinality of the residue 
field of any prime ideal of $\mathcal O_K$ above $p$ and the ramification index of $p$  in $\mathcal{O}_K$, respectively. We also set 
\[w_{q_p}(n)=\sum_{j\geq 1}\left\lfloor \frac{n}{q^j_p}\right\rfloor\]
and define for every integer $n\geq 1$, 
\[\omega_p(n)=\omega_{K,p}(n):={\left\lfloor \frac{w_{q_p}(n)}{e_p}\right\rfloor}.\]

\begin{theorem}\label{Thm-Intro2}
Suppose that $K/\Q$ is a Galois extension which is tamely ramified at each prime. Let 
$\{f_n(X)\}_{n\geq 0}$ be a $\Z$-basis of $\Int_{\Q}(\mathcal O_K)$ such that $\deg(f_n)=n$, for each $n\in\N$. 
Then we can write 
\[f_n(X)=\frac{g_n(X)}{\prod_{p}p^{\omega_p(n)}} \]
for some monic polynomial $g_n(X)$ in $\Z[X]$, 
where the product is over all primes $p$ of $\Z$. 
\end{theorem} 

The proof of the above theorem is constructive: first, we construct a basis of $\Int_{\Q_p}(\mathcal O_{K_{\p}})$,  for any prime ideal $\p$ of $\mathcal O_K$ above the rational prime $p$, from knowledge of a local basis of $\Int(\mathcal O_{K_\p})$ (here and for the rest of the paper, $K_{\p}$ denotes the $\p$-adic completion of $K$); then, we use the Chinese Remainder Theorem to construct a global basis of $\Int_\Q(\mathcal O_K)$. 

\section{A characterization of Galois extensions} \label{section2}

We introduce the following general notation, extending that of the introduction. 
Let $D$ be an integral domain with quotient field $K$ and let $A$ be a torsion-free $D$-algebra.
Let $B:=A\otimes_DK$  be the extended $K$-algebra; we have canonical embeddings $A\hookrightarrow B$ and $K\hookrightarrow B$. 
For $a\in A$ and $f\in K[X]$, the value $f(a)$ belongs to $B$, and the following definition makes sense (see also \cite{PerWer}):
\[\Int_{K}(A):=\{f\in K[X]\mid f(a)\in A, \forall a\in A\}.\]
Clearly, $\Int_{K}(A)$ is a $D$-algebra. It is easy to see that $\Int_K(A)$ is contained in the classical ring of integer-valued polynomials $\Int(D)=\{f\in K[X] \mid f(D)\subseteq D\}$ if and only if $A\cap K=D$, and this will be the case henceforth.

A sequence of polynomials $\{f_n(X)\}_{n\in\N}\subset\Int_K(A)$ which forms a basis of $\Int_K(A)$ as a $D$-module and such that $\deg(f_n)=n$ for each $n\in\N$, is called a \emph{regular basis} of $\Int_K(A)$. 
We define $\mathfrak I_n\left(\Int_{K}(A)\right)$ to be the $D$-module generated by the 
leading coefficients of all the polynomials $f\in \Int_{K}(A)$ of degree exactly $n$; we call these $D$-modules \emph{characteristic ideals}. For each $n\in\N$, by the above assumption and \cite[Proposition II.1.1]{CaCh}, $\mathfrak I_n(\Int_{K}(A))$ is a fractional ideal of $D$. Moreover, the set of characteristic ideals forms an ascending sequence:
$$D\subseteq \mathfrak I_0(\Int_{K}(A))\subseteq\ldots\subseteq \mathfrak I_n(\Int_{K}(A))\subseteq\mathfrak I_{n+1}(\Int_{K}(A))\subseteq\ldots\subseteq K.$$
The link between regular bases and characteristic ideals is given by \cite[Proposition II.1.4]{CaCh}, which says that a sequence of polynomials $\{f_n(X)\}_{n\in\N}$ of $\Int_{K}(A)$ is a regular basis if and only if, for each $n\in\N$, $f_n(X)$ is a polynomial of degree $n$ whose leading coefficient generates $\mathfrak{I}_n(\Int_{K}(A))$ as a $D$-module.
In particular, note that $\Int_\Q(\cO_K)$ and $\Int_{\Q_p}(\cO_K)$ (for $K/\Q$ and $K/\Q_p$ finite field extension) each admits a regular basis.

We fix from here to the end of this section a number field $K$ and denote by 
$\mathcal O_K$ its ring of algebraic integers. For any prime ideal $\p$ of $\mathcal O_K$, 
we denote $\mathcal O_{K,(\p)}$ the localization of $\mathcal O_K$ at $\p$, \emph{i.e.}, the localization at the multiplicative set $\mathcal O_K\setminus \p$. Moreover, for any $\Z$-module $M$ and any prime number $p$, 
we denote $M_{(p)}$ the localization at $p$, \emph{i.e.}, the 
localization at the multiplicative set $\Z\setminus p\Z$. We also denote by $K_\p$ the completion of $K$ at $\p$ and by $\mathcal O_{K,\p}$ the valuation ring of $K_{\p}$.

\begin{proposition} \label{prop2.1}
We have $\Int_\Q(\mathcal O_K)=\bigcap_\p\Int_\Q(\mathcal O_{K,(\p)})$ and 
\[\mathfrak I_n(\Int_\Q(\mathcal O_{K}))=\bigcap_{\p}\mathfrak I_n(\Int_\Q(\mathcal O_{K,(\p)})\]
 for each $n\in\N$, where the intersection is over all  prime ideals  of $\mathcal O_K$. 
\end{proposition} 

\begin{proof}
We first observe that $\Int_{\Q}(\mathcal O_K)=\bigcap_p\Int_\Q(\mathcal O_K)_{(p)}$; here the intersection 
is over all primes of $\Z$.  Then one observes that 
$\Int_\Q(\mathcal O_K)_{(p)}=\Int_{\Q}(\mathcal O_{K,(p)})$ (see for example \cite{Wer}). We conclude 
that $\mathfrak I_n(\Int_\Q(\mathcal O_{K,\p})_{(p)})$ is equal to $\mathfrak I_n(\Int_\Q(\mathcal O_{K,(p)})$, 
showing the second part. Further, $\mathcal O_{K,(p)}=\bigcap_{\p\mid p}\mathcal O_{K,(\p)}$, where 
$\mathcal O_{K,(\p)}$ is the localization of $\mathcal O_K$ at $\p$, and the intersection is over all 
prime ideals $\p$ of $\mathcal{O}_K$ which lie above $p$. Therefore 
\begin{equation}\label{int}
\Int_\Q(\mathcal O_{K,(p)})=\bigcap_{\p\mid p}\Int_\Q(\mathcal O_{K,(\p)})\end{equation}
and the result follows. \qed
\end{proof}

\begin{remark}
Note that, if $K/\Q$ is Galois, then $\Int_\Q(\mathcal O_{K,(\p)})$, for $\p\mid p$, are all equal because $\Gal(K/\Q)$ acts transitively on the set of rings $\{\mathcal O_{K,(\p)}:\p\mid p\}$.  Therefore \eqref{int} reads as 
\[\Int_\Q(\mathcal O_{K,(p)})=\Int_\Q(\mathcal O_{K,(\p)})\]
for each $\p\mid p$.  A similar  argument has been used in \cite[Proposition 1.10]{PerFinite}.
\end{remark}

In order to determine some relations of containments between the rings $\Int_{\Q}(\mathcal O_{K,(\p)})$, we introduce the following  object: given an extension of commutative rings $R\subseteq S$, we consider the null ideal of $S$ over $R$, that is, $N_R(S)=\{g\in R[X]\mid g(S)=0\}\subseteq R[X]$  (for results connected to null ideals see for example \cite{PerDivDiff, PerWerProperly, Wer}).
\begin{proposition}\label{null ideals}
Let $K$ be a number field and let $p\in\Z$ be a prime. Let $\p\subset \mathcal O_K$ be a prime ideal above $p$ with 
ramification index $e$ and residue class degree $f$. Then
$$N_{\mathbb{F}_p}(\mathcal O_K/\p^{e})=((X^{p^f}-X)^e)$$
\end{proposition}
\begin{proof}
Since $\pi:\mathcal O_K/\p^{e}\twoheadrightarrow \mathcal O_K/\p^{e-1}\twoheadrightarrow\ldots\twoheadrightarrow \mathcal O_K/\p\cong\mathbb{F}_{p^f}$ and $\mathbb{F}_p$ embeds in all of these rings (because $\p^{i}\cap\Z=p\Z=\p\cap\Z$, for all $i=1,\ldots,e$) we have
$$\xymatrix{
\mathcal O_K/\p^{e}\ar@{>>}[r]&\mathcal O_K/\p^{e-1}\ar@{>>}[r]&\ldots\ar@{>>}[r]&\mathcal O_K/\p\\
&&&\mathbb{F}_p\ar@{^{(}->}[u]\ar@{^{(}->}[ull]\ar@{^{(}->}[ulll]
}$$
so, in particular, we have the following chain of containments between these ideals of $\mathbb{F}_p[X]$:
$$N_{\mathbb{F}_p}(\mathcal O_K/\p^{e})\subseteq N_{\mathbb{F}_p}(\mathcal O_K/\p^{e-1})\subseteq\ldots\subseteq N_{\mathbb{F}_p}(\mathcal O_K/\p).$$
Since $\mathcal O_K/\p$ is a finite field with $p^f$ elements, the ideal $N_{\mathbb{F}_p}(\mathcal O_K/\p)$ is generated by the polynomial 
$X^{p^f}-X$. The proof proceeds by induction on $e$. Suppose that $N_{\mathbb{F}_p}(\mathcal O_K/\p^{e-1})$ is generated by $(X^{p^f}-X)^{e-1}$. It is easy to see that $(X^{p^f}-X)^e$ is contained in $N_{\mathbb{F}_p}(\mathcal O_K/\p^{e})$. Therefore, the latter ideal is generated by a polynomial $g\in\mathbb{F}_p[X]$ which is zero on all the elements of $\mathcal O_K/\p^e$ of the form
$$g(X)=(X^{p^f}-X)^{e-1}h(X)=F_q(X)^{e-1} \prod_{\gamma\in S}(X-\gamma)$$
for some $S\subseteq\mathbb{F}_{p^f}=\mathbb{F}_q$. Suppose that $S$ is strictly contained in $\mathbb{F}_{q}$ and let $\overline{\gamma}\in\mathbb{F}_{q}\setminus S$. Without loss of generality, we may assume that $\overline{\gamma}=0$ (apply the automorphism $X\mapsto X-\gamma$, if necessary; this is an automorphism for $\mathbb{F}_{p^f}$ and $\mathcal O_K/\p^e$).

Let $t\in P/P^e\subset \mathcal O_K/\p^e$ such that its index of nilpotency is $e$ (that is, $t^e=0$ but $t^{e-1}\not=0$). Then $F_q(t)^{e-1}=t^{e-1}\cdot(t^{q-1}-1)^{e-1}$ is not zero in $O_K/\p^e$, because $t^{q-1}-1$ is a unit of $O_K/\p^e$ (because $\p/\p^e$ is the Jacobson radical of $O_K/\p^e$).

In the same way, $h(t)=\prod_{\gamma\in S}(t-\gamma)$ is not in the kernel of $\pi:\mathcal O_K/\p^{e}\twoheadrightarrow \mathcal O_K/\p^{e-1}$, which is $\p/\p^e$, because modulo $\p$, $h(t)$ is not zero ($\pi(h(t))=h(\pi(t))=h(0)\not=0$, because $0\notin S$). Hence, $h(t)$ is invertible, so that $g(t)=F_q(t)^{e-1}\cdot h(t)$ is not zero, contradiction.
\qed
\end{proof}

\begin{proposition}\label{comparison IntQrings}
Let $K,K'$ be number fields, with prime ideals $\p,\p'$ of residual characteristic $p$, respectively, and with ramification index and residue class degree equal to $e,f$ and $e',f'$, respectively.  Suppose that 
$$\Int_{\Q}(\mathcal O_{K',(\p')})\subseteq\Int_{\Q}(\mathcal O_{K,(\p)})$$
Then $\p\cap\Z=p\Z=\p'\cap\Z$, $f|f'$ and $e\leq e'$.  In particular, if the above containment is an equality, we have that $\p\cap\Z=p\Z=\p'\cap\Z$, $f=f'$ and $e= e'$.
\end{proposition}
\begin{proof}
Suppose that $\p\cap\Z=p\Z$ and $\p'\cap\Z=p'\Z$. Observe that 
\[\Int_{\Q}(\mathcal O_{K,(\p)})\cap \Q=(\Int(\mathcal O_{K,(\p)})\cap K)\cap \Q=\mathcal O_{K,(\p)}\cap\Q=\Z_{(p)}\]
and analogously for $\p'$ and $p'$. Therefore 
\[\Int_{\Q}(\mathcal O_{K',(\p')})\cap\Q=\Z_{(p')}\subseteq\Int_{\Q}(\mathcal O_{K,(\p)})\cap\Q=\Z_{(p)},\]
so that $p=p'$. 

By Proposition \ref{null ideals}, the containment of the hypothesis implies that
\begin{equation}\label{1}
\frac{(X^{p^{f'}}-X)^{e'}}{p}\in \Int_{\Q}(\mathcal O_{K,(\p)}).
\end{equation} 
In particular, modulo $p$, we have 
$$(X^{p^{f'}}-X)^{e'}\in N_{\mathbb{F}_p}(\mathcal O_K/\p^{e})=((X^{p^f}-X)^e),$$
again by Proposition \ref{null ideals}. It follows that $(X^{p^{f'}}-X)^{e'}\in (X^{p^f}-X)$ and since the latter is a radical ideal (because $X^{p^f}-X$ is a separable polynomial), this means that $X^{p^{f'}}-X$ belongs to $(X^{p^f}-X)$ which is equivalent to $\mathbb{F}_{p^f}\subseteq\mathbb{F}_{p^{f'}}$ which holds if and only if $f|f'$, as claimed.

In the same way, since $X^{p^{f'}}-X$ is a separable polynomial (every irreducible factor appears with multiplicity $1$ in the factorization of $X^{p^{f'}}-X$ over $\mathbb{F}_p$), we deduce that $e\leq e'$.\qed
\end{proof}

\vskip0.5cm
We recall that, by a result of Gerboud (see \cite{Gerb} and also \cite[Prop. IV.3.3]{CaCh}) we have
\begin{equation}\label{Gerb}
\Int(\Z_{(p)},\mathcal O_{K,(\p)})=\{f\in K[X] \mid f(\Z_{(p)})\subseteq \mathcal O_{K,(\p)}\}=\Int(\Z_{(p)})\cdot \mathcal O_{K,(\p)}
\end{equation}
%(where $\Int(\Z_{(p)},O_{K,\p})\subset K[X]$ and $\Int(\Z_{(p)})\subset\Q[X]$).
\vskip0.3cm
\begin{lemma}\label{IntQOKPef1}
Let $K$ be a number field and let $\p\subset \mathcal O_K$ be a prime ideal which lies above a prime $p\in\Z$. Let $e=e(\p|p)$ and $f=f(\p|p)$ be the ramification index and residue class degree, respectively. Then the following conditions are equivalent:
 \begin{itemize}
  \item[i)] $\Int(\Z_{(p)})\subseteq \Int(\mathcal O_{K,(\p)})$.
  \item[ii)] $\Int(\Z_{(p)},\mathcal O_{K,(\p)})=\Int(\mathcal O_{K,(\p)})$ % ($\Leftrightarrow \Z_{(p)}$ is polynomially dense in $O_{K,\p}$)
  \item[iii)] $\Int_{\Q}(\mathcal O_{K,(\p)})=\Int(\Z_{(p)})$.  
  \item[iv)] $e=f=1$.
\end{itemize}
If any of this equivalent conditions holds, then
\begin{align*}
\Int(\Z_{(p)})\cdot \mathcal O_{K,(\p)}&=\Int(\mathcal O_{K,(\p)}).
\end{align*}

\end{lemma}
\begin{proof} Obviously, conditions i) and iii) are equivalent, since $\Int_{\Q}(\mathcal O_{K,(\p)})$ is always contained in $\Int(\Z_{(p)})$. 

If i) holds, then by (\ref{Gerb}) above we have $\Int(\Z_{(p)},\mathcal O_{K,(\p)})\subseteq\Int(\mathcal O_{K,(\p)})$, which is condition ii), since we always have the containment $\Int(\Z_{(p)},\mathcal O_{K,(\p)})\supseteq\Int(\mathcal O_{K,(\p)})$. Conversely, if condition  ii) holds, then again by (\ref{Gerb}) above we have $\Int(\Z_{(p)})\subseteq \Int(\mathcal O_{K,(\p)})$.

The equivalence between iii) and iv) follows immediately from Proposition \ref{comparison IntQrings}. \qed
\end{proof}

\begin{corollary}\label{ptotallysplit}
Let $K$ be a number field and let $p\in\Z$ be a prime. Then the following conditions are equivalent:
\begin{itemize}
\item[i)] $\Int(\Z_{(p)})=\Int_{\Q}(\mathcal O_{K,(p)})$.
\item[ii)] $p$ is totally split in $\mathcal O_K$.
\item[iii)] $\frac{X^p-X}{p}\in \Int_{\Q}(\mathcal O_{K})$.
\end{itemize}
\end{corollary}
\begin{proof}
The proof of the equivalence i)$\Leftrightarrow$ii) follows immediately from (\ref{int}) and Lemma \ref{IntQOKPef1}. Indeed, if $p$ is totally split in $\mathcal O_K$ then, for each prime ideal $\p$ of $\mathcal O_K$ above $p$, we have $\Int(\Z_{(p)})=\Int_{\Q}(\mathcal O_{K,(\p)})$, so that by (\ref{int}) we have the equality $\Int(\Z_{(p)})=\Int_{\Q}(\mathcal O_{K,(p)})$. Conversely, if the last equality holds, then by (\ref{int}), for each prime ideal $\p$ of $\mathcal O_K$ above $p$, we have $\Int(\Z_{(p)})\subseteq\Int_{\Q}(\mathcal O_{K,(\p)})\subseteq\Int(\Z_{(p)})$, so equality holds throughout and $p$ is totally split in $\mathcal O_K$.

%If $f\in \Int_{\Q}(\mathcal O_{K,(p)})$, then $\Int(\Z_{(p)})=\Int_{\Q}(\mathcal O_{K,(p)})$, since $\Int(\Z_{(p)})$ is generated, as a $\Z_{(p)}$-module, by $f(X)$ and its compositions  (\cite[Chapter II]{CaCh}).

We show now that ii)$\Rightarrow$iii). Suppose that $p$ is totally split in $\mathcal O_K$, so that, by the Chinese Remainder Theorem we have
$$\mathcal O_K/p\mathcal O_K\cong\mathbb{F}_p^n$$
where $n=[K:\Q]$. Hence, $X^p-X$ is zero on $\mathcal O_K/p\mathcal O_K$, so that $f(X)=\frac{X^p-X}{p}$ is in $\Int_{\Q}(\mathcal O_K)$. Conversely, suppose that $f(X)$ is in $\Int_{\Q}(\mathcal O_K)$. Then $X^p-X$ is zero on $\mathcal O_K/p\mathcal O_K\cong\prod_{i=1}^g \mathcal O_K/\p_i^{e_i}$, where $\p_1,\ldots,\p_g$ are the prime ideals of $O_K$ above $p$, with ramification index $e_i=e(\p_i|p)$ and residue class degree $f_i=f(\p_i|p)$. Consequently, $X^p-X$ is zero on each factor ring $\mathcal O_K/\p_i^{e_i}$, for $i=1,\ldots,g$. Let $\overline{\alpha}$ be in the Jacobson ideal of $\mathcal O_K/\p_i^{e_i}$, that is, $\overline{\alpha}$ is in $\p_i/\p_i^{e_i}$ (the unique maximal ideal of $\mathcal O_K/\p_i^{e_i}$). Then $1-\overline{\alpha}^{p-1}$ is a unit in $\mathcal O_K/\p_i^{e_i}$. But by assumption $\overline{\alpha}^p-\overline{\alpha}=\overline{\alpha}(\overline{\alpha}^{p-1}-1)=0$, so that $\overline{\alpha}=0$. Therefore, $\mathcal O_K/\p_i^{e_i}$ has trivial Jacobson ideal, which happens precisely when $e_i=1$. If $f_i>1$, then $\mathcal O_K/\p_i$ is a proper finite field extension of $\mathbb{F}_p$, so if we take an element $\overline{\gamma}$ of $\mathcal O_K/\p_i\setminus\mathbb{F}_p$, $\overline{\gamma}$ will be a zero of a monic irreducible polynomial $q(X)$ over $\mathbb{F}_p$ of degree strictly larger than $1$. Since $X^p-X$ is zero on $\overline{\gamma}$, we would have that $q(X)$ divides $X^p-X$ over $\mathbb{F}_p$, which is clearly not possible because $X^p-X$ splits over  $\mathbb{F}_p$. This shows that iii)$\Rightarrow$ii).\qed
\end{proof}

The next result characterizes the finite Galois extensions of $\Q$ in terms of the rings $\Int_{\Q}(\mathcal O_K)$. In particular, we can  recover $\mathcal O_K$ from $\Int_{\Q}(\mathcal O_K)$, if $K/\Q$ is Galois. Given a subring $R$ of $\Q[X]$, for each $\alpha\in\overline{\Z}$ we consider the following subset of $\Q(\alpha)$:
$$R(\alpha)=\{f(\alpha) \mid f\in R\}$$

\begin{theorem}\label{coro2.9}\label{equalIntQOK}
Let $K/\Q$ be a finite extension and let $R_K=\Int_{\Q}(\mathcal O_K)$. Then
\[K/\Q \textnormal{ is a Galois extension } \Leftrightarrow\{\alpha\in\overline{\Z} \mid R_K(\alpha)\subset\overline{\Z}\}=\mathcal O_K.\]
In particular, if $K$ and $K'$ are two Galois extensions of $\Q$ such that 
$\Int_{\Q}(\mathcal O_K)=\Int_{\Q}(\mathcal O_{K'})$, then $K=K'$.
\end{theorem}
Note that the condition $R_K(\alpha)\subset\overline{\Z}$ is equivalent to $R_K(\alpha)\subseteq \mathcal O_{\Q(\alpha)}$.
\begin{proof}
The second statement about $K$ and $K'$ follows immediately from the first. 

For the first statement, let $R_K=\Int_{\Q}(\mathcal O_K)$ and 
suppose that $\{\alpha\in\overline{\Z} \mid R_K(\alpha)\subset\overline{\Z}\}=\mathcal O_K$. It is easily seen that the left-hand side is  invariant under the action of the absolute Galois group ${\rm Gal}(\overline{\Q}/\Q)$. Hence, $\mathcal O_K$ contains the ring of integers of all the conjugates of $K$ over $\Q$, so $K/\Q$ is Galois. 

Conversely, suppose that $K/\Q$ is a Galois extension. It is clear that we have the containment $\{\alpha\in\overline{\Z} \mid R_K(\alpha)\subset\overline{\Z}\}\supseteq\mathcal O_K$. Conversely, let $\alpha\in\overline{\Z}$, $\alpha\notin \mathcal O_K$. We have to show that there exists $f\in \Int_{\Q}(\mathcal O_K)$ such that $f(\alpha)\notin\overline{\Z}$. Let $K_{\alpha}=\Q(\alpha)$ and let $N_{\alpha}$ be the Galois closure of $K_{\alpha}$ over $\Q$ (the compositum inside $\overline{\Q}$ of all the conjugates over $\Q$ of $K_{\alpha}$). We have that $\alpha\notin K\Leftrightarrow K_{\alpha}\not\subset K\Leftrightarrow N_{\alpha}\not\subset K$, where the last equivalence holds because by assumption $K/\Q$ is Galois.

By Tchebotarev's Density Theorem, a Galois extension $K$ of $\Q$ is completely determined by the set of primes $S(K/\Q)$ which are totally split in $K$ (see \cite[Chapter VII, Corollary 13.10]{Neu}). Hence, the condition $N_{\alpha}\not\subset K$ is equivalent to $S(K/\Q)\not\subset S(N_{\alpha}/\Q)$, that is, the set of primes $p\in\Z$ which are totally split in $K$ is not contained in the set of primes which are totally split in $N_{\alpha}$. Let $p\in\Z$ be such a prime and suppose also that 
\begin{itemize}
\item[-] $p$ is ramified neither in $K$ nor in $N_{\alpha}$.
\item[-] $p$ does not divide $[\mathcal O_{K_{\alpha}}:\Z[\alpha]]$
\end{itemize}
The above primes are always finite in number and since the above set is infinite, by removing the latter primes we still get a non-empty set. By Corollary \ref{ptotallysplit}, $f(X)=\frac{X^p-X}{p}$ is in $\Int_{\Q}(\mathcal O_K)$ but not in $\Int_{\Q}(\mathcal O_{N_{\alpha}})$. Recall that a prime $p\in\Z$ splits completely in the normal closure $N_{\alpha}$ of $K_{\alpha}$ (over $\Q$) if and only if it splits completely in $K_{\alpha}$ (\cite[Chapt. 4, Corollary of Theorem 31]{Marcus}). Hence, there exists some prime ideal $\p$ of  $\mathcal O_{K_{\alpha}}$ above $p$ which has inertia degree strictly greater than $1$. Since $p$ does not divide $[\mathcal O_{K_{\alpha}}:\Z[\alpha]]$, it follows by Dedekind-Kummer's Theorem (see \cite[Chapter I, Proposition 8.3]{Neu}) that the factorization in $\mathbb{F}_p[X]$ of the residue modulo $p$ of the minimal polynomial $p_{\alpha}(X)$ of $\alpha$ over $\Z$ has at least one irreducible polynomial over $\mathbb{F}_p$ whose degree is strictly greater than $1$; this factor corresponds to a prime ideal $\p$ of $\mathcal O_{K_{\alpha}}$ above $p$ which is not inert, that is $\mathcal O_{K_{\alpha}}/\p\supsetneq\mathbb{F}_p$. In particular, this means that modulo $\p$, $\alpha$ is not in $\mathbb{F}_p$, and so it is not annihilated by $\overline{g}(X)=X^p-X$ (equivalently, modulo $\p$, $\alpha$ is a zero of an irreducible polynomial over $\mathbb{F}_p$ of degree strictly greater than $1$). This implies that $f(\alpha)$ is not integral over $\Z$.\qed
\end{proof}

\begin{remark}\label{alternative proof}
We also offer a shorter proof of the second statement in Theorem \ref{equalIntQOK} based on the followin claim: let $K,K'$ be two finite Galois extensions over $\Q$. Then
$$\mathcal{O}_K\subseteq\mathcal{O}_{K'}\Leftrightarrow \Int_{\Q}(\mathcal{O}_{K'})\subseteq\Int_{\Q}(\mathcal{O}_{K})$$
We prove the implication $(\Leftarrow)$, the other being obvious (and it is true even without the Galois assumption). Suppose then that $\Int_{\Q}(\mathcal{O}_{K'})\subseteq\Int_{\Q}(\mathcal{O}_{K})$ and let $p\in\Z$ be a prime which is totally split in $\mathcal O_{K'}$. Then by Corollary 2.6 we have
$$\Int(\Z_{(p)})=\Int_{\Q}(O_{K',(p)})\subseteq\Int_{\Q}(O_{K,(p)}).$$
Since in any case $\Int_{\Q}(O_{K,(p)})$ is contained in $\Int(\Z_{(p)})$, the containment  in the above equation is an equality, so that again by Corollary 2.6 we have that $p$ is totally split in $K$, too. This shows that the set of primes $p\in\Z$ which are totally split in $\mathcal O_{K'}$ is contained in the set of primes which are totally split in $\mathcal O_{K}$. Therefore, by \cite[Chapt. VII, Prop. 13.9]{Neu} it follows that $K\subseteq K'\Leftrightarrow \mathcal{O}_K\subseteq \mathcal{O}_{K'}$.

Note that the above statement shows in particular that the functor $\Int_{\Q}:\mathcal G\to \mathcal C$ is full, as we claimed in the Introduction.
%\vskip1.5cm
%
%If $K$ and $K'$ are two Galois extensions of $\Q$ such that 
%$\Int_{\Q}(\mathcal O_K)=\Int_{\Q}(\mathcal O_{K'})$, then $K=K'$.
%
%Suppose the above assumption is satisfied. In particular, for each prime $p\in\Z$, if we localize at $\Z\setminus p\Z$ we have the following equality:
%\begin{equation}\label{IntQOKpIntQOK'p}
%\Int_{\Q}(\mathcal O_{K,(p)})=\Int_{\Q}(\mathcal O_{K',(p)})
%\end{equation}
%Let now $p\in\Z$ be a prime which is totally split in $\mathcal O_K$. Then the left hand side of (\ref{IntQOKpIntQOK'p}) is equal to $\Int(\Z_{(p)})$, by Corollary \ref{ptotallysplit}. By the same Corollary, $p$ is totally split in $\mathcal O_{K'}$. Symmetrically, if $p$ is totally split in $K'$ we deduce in the same way that $p$ is totally split in $K$. Therefore, the sets of primes $p\in\Z$ which are totally split in the Galois extensions $K$ and $K'$, respectively, coincide. By the Tchebotarev Density Theorem (see \cite[Chapt. VII, \S 13, Corollary 3.10]{Neu}), a finite Galois extension $K$  is uniquely determined by the set of primes $p\in\Z$ which are totally split in $\mathcal O_K$, so $K=K'$.
\end{remark}

\section{Characteristic ideals} \label{sec3}

Proposition \ref{prop2.1} reduces the study of characteristic ideals of $\Int_\Q(\mathcal O_K)$ to the study of 
characteristic ideals in the local case. We will address a description of these ideals and apply the local results to 
the global context. 

\subsection{Local case} 
Fix a finite field extension $K/\Q_p$ having residue class degree $f$ and ramification degree $e$. 
Denote by $v_p$ the $p$-adic valuation of $\Q_p$, normalized such that 
$v_p(p)=1$. Let:
\[w_{p}(n):=v_p(n!)=\sum_{j\geq 1} \left\lfloor\frac{n}{p^j}\right\rfloor\]
and, if $q=p^f$ is the cardinality of the residue field of $K$, put 
\[w_{q}(n):=\sum_{j\geq 1} \left\lfloor\frac{n}{q^j}\right\rfloor.\]
The following equality follows from \cite[Corollary II.2.9]{CaCh}: 
\[-v_p\left(\mathfrak I_n\left(\Int(\Z_p)\right)\right)=w_p(n)\]
and, similarly, we have:
\begin{equation}\label{reg-basis}
-v_{\pi}\left(\mathfrak I_n\left(\Int(\mathcal O_K)\right)\right)={w_q(n)}
\end{equation}
where $\pi$ is a uniformizer of $K$ and $v_{\pi}$ the associated valuation. 

We define finally 
\[w_{\mathcal O_K}^{\Q_p}(n):=-v_p\left(\mathfrak I_n\left(\Int_{\Q_p}(\mathcal O_K)\right)\right).\]
The following equality holds because of the next lemma, noticing that $\left\lceil-\frac{n}{e}\right\rceil=-\left\lfloor\frac{n}{e}\right\rfloor$:
$$\mathfrak I_n(\Int(\mathcal O_K))\cap \Q_p=p^{-\lfloor \frac{w_q(n)}{e}\rfloor}\Z_p$$
and since $\mathfrak I_n(\Int_{\Q_p}(\mathcal O_K))\subseteq\mathfrak I_n(\Int(\mathcal O_K))\cap \Q_p$, for every $n\in\N$ we have:
\begin{equation}\label{inequality}
w_{\mathcal O_K}^{\Q_p}(n)\leq \left\lfloor \frac{w_q(n)}{e}\right\rfloor .
\end{equation}

\begin{lemma}
Let $n\in\Z$ and $e=e(\p|p)$, where $\p$ is the maximal ideal of $\mathcal O_K$. Then 
\[\p^n\cap\Q_p=p^{\lceil\frac{n}{e}\rceil}\Z_p.\]
\end{lemma}
\begin{proof}
$(\supseteq)$. Clearly, $p^{\lceil\frac{n}{e}\rceil}\in \p^n\Leftrightarrow v_\p(p^{\lceil\frac{n}{e}\rceil})=e\cdot\lceil\frac{n}{e}\rceil\geq n$, which is true, so the containment follows, since clearly $\p^n\cap\Q_p$ is a $\Z_p$-module.

$(\subseteq)$. Let $\alpha\in \p^n\cap\Q_p$, say $\alpha=p^m u$, where $u\in \Z_p^*$ and $m=v_p(\alpha)$. Then $v_\p(\alpha)=me$ which has to be greater than or equal to $n$. Therefore, $m\geq\lceil\frac{n}{e}\rceil$, so $\alpha\in p^{\lceil\frac{n}{e}\rceil}\Z_p$.\qed
\end{proof}

The main result of this section shows the opposite inequality in (\ref{inequality}) in the case of tame ramification for a finite Galois extension. By the above remarks, this corresponds to say that $\mathfrak I_n(\Int_{\Q_p}(\mathcal O_K))=\mathfrak I_n(\Int(\mathcal O_K))\cap \Q_p$, for each $n\in\N$. We  show in Examples
\ref{remark3.5} that these two conditions, namely, Galois and tame ramification, cannot be relaxed. 

\begin{theorem}\label{equality-theorem}
Let $K/\Q_p$ be a finite tamely ramified Galois extension, with ramification index $e$ and residue field of cardinality $q$.
Then for all $n\in\N$ we have 
$$w_{\mathcal O_K}^{\Q_p}(n)=\left\lfloor \frac{w_q(n)}{e}\right\rfloor.$$
In particular, $w_{\mathcal O_K}^{\Q_p}(n)$ only depends on $n$, $q$ and $e$. 
\end{theorem}
\begin{proof} 
By (\ref{inequality}) it is sufficient to show that $d=p^{-\lfloor \frac{w_q(n)}{e}\rfloor}$ is in $\mathfrak{I}_n(\Int_{\Q_p}(\mathcal O_K))$. 

We observe that if $f(X)=\sum_{i=0}^na_iX^i$ belongs to $\Int(\mathcal O_K)$, then 
$f^\sigma(X):=\sum_{i=0}^n\sigma (a_i)X^i$ belongs to $\Int(\mathcal O_K)$ 
for all $\sigma\in G=\Gal(K/\Q_p)$ (here we use crucially the assumption that $K/\Q_p$ is Galois). 
As a consequence, if we denote $\tr=\tr_{K/\Q_p}: K\rightarrow\Q_p$ the trace homomorphism, we see that  
\[\mathrm{Tr}(f):=\sum_{\sigma\in G} f^{\sigma}=\sum_{i=0}^n\tr(a_i)X^i\] 
belongs to $\Int_{\Q_p}(\mathcal O_K)$, if $f\in\Int(\mathcal O_K)$. Therefore, the trace homomorphisms between the function fields $\mathrm{Tr}:K(X)\to \Q_p(X)$ restricts to $\mathrm{Tr}:\Int(\mathcal O_K)\to\Int_{\Q_p}(\mathcal O_K)$. 

Since $p\nmid e$, the trace homomorphism $\tr$ is surjective (the converse is also true, see \cite[Chapter 5, Corollary, p. 227]{Nar}). Fix $\alpha\in \mathcal{O}_K$ such that $\tr(\alpha)=1$ and let $c=d\alpha\in K$. In particular, since the trace is a $\Q_p$-homomorphism, we have $\tr(c)=d$. Note that the $v_{\pi}$-value of $c$ is greater than or equal to $-e\lfloor \frac{w_q(n)}{e}\rfloor\geq -w_q(n)$. By (\ref{reg-basis}), $c$ is in $\mathfrak{I}_n(\Int(\mathcal{O}_K))$, so there exists $f\in \Int(\mathcal{O}_K)$ of degree $n$ whose leading coefficient is equal to $c$. 
Therefore, $\mathrm{Tr}(f)$ is a polynomial of degree $n$ in $\Int_{\Q_p}(\mathcal{O}_K)$ with leading coefficient equal to $d$, as we wanted to show.\qed
\end{proof}

\begin{remark} We remark that, from the fact that $\tr=\tr_{K/\Q_p}:\mathcal O_K\to\Z_p$ is surjective (because the extension is tame), the proof of Theorem \ref{equality-theorem} also shows that the restriction of the trace homomorphism $\mathrm{Tr}:\Int(\mathcal O_K)\to\Int_{\Q_p}(\mathcal O_K)$ is surjective. In fact, for each $n\in\N$, the $n$-th element of a regular basis of $\Int_{\Q_p}(\mathcal O_K)$, whose leading coefficient has $p$-adic value $-\left\lfloor \frac{w_q(n)}{e}\right\rfloor$ by the above theorem, is the image via the trace homomorphism of a polynomial of $\Int(\mathcal O_K)$. 

Obviously, if $\mathrm{Tr}$ is surjective, it is easily seen that $\tr$ is surjective, because  $\Z_p$ is contained in $\Int_{\Q_p}(\mathcal O_K)$. Finally, we have the following commutative diagram:
\[
\xymatrix{
\Int(\mathcal O_K)\ar[r]^{\mathrm{Tr}}&\Int_{\Q_p}(\mathcal O_K)\\
\mathcal{O}_K\ar@{^{(}->}[u]\ar[r]^{\tr}&\Z_p\ar@{^{(}->}[u]
}
\]

\end{remark}
The next corollary shows that  Theorem \ref{equalIntQOK} is false in the local case.
\begin{corollary} 
Let $K_1,K_2$ be two finite tamely ramified Galois extensions of $\Q_p$. Then $\Int_{\Q_p}(\mathcal{O}_{K_1})=\Int_{\Q_p}(\mathcal{O}_{K_2})$ if and only if $K_1$ and $K_2$ have the same ramification index and residue field degree.
\end{corollary}
\begin{proof}
Suppose that $K_1$ and $K_2$ have the same ramification index and residue field degree. In particular, the functions $w_{\mathcal O_{K_i}}^{\Q_p}(n)$, for $i=1,2$, are the same, by Theorem \ref{equality-theorem}. Hence, by definition, the set of characteristic ideals of the rings $\Int_{\Q_p}(\mathcal{O}_{K_i})$, $i=1,2$, coincide, so these rings have a common regular bases, and therefore they are equal.

Conversely, if the $\Int_{\Q_p}$-rings are equal, a straightforward adaptation of Proposition \ref{comparison IntQrings} to the present setting shows that the ramification indexes and residue field degrees of $K_1$ and $K_2$ are the same. Note that this part of the proof holds also without the tameness assumption.\qed
\end{proof}

\begin{remark}\label{unramified explicit bases} In the case $K/\Q_p$ is a finite unramified extension (so, in particular, a Galois extension), we can given an explicit basis of $
\Int_{\Q_p}(\mathcal{O}_K)$. Let $q=p^f$ be the cardinality of the residue field of $\mathcal O_K$. By Theorem \ref{equality-theorem}, for all $n\in\N$ we have $w_{\mathcal O_K}^{\Q_p}(n)=w_q(n)$. Let 
\[f(X):=\frac{X^q-X}{p}\]
which clearly belongs to $\Int_{\Q_p}(\mathcal O_{K})$. For $k\in\N$, we denote by $f^{\circ k}(X)$ the composition of $f$ with itself $k$ times, namely $f^{\circ k}(X)=f\circ\ldots\circ f(X)$. If $k=0$ we put $f^{0}(X):=X$.
For each positive integer $n\in\N$, we consider its $q$-adic expansion:
$$n=n_0+n_1q+\ldots+n_rq^r$$
where $n_i\in\{0,\ldots,q-1\}$ for all $i=0,\ldots,r$. We define
$$f_n(X):=\prod_{i=0}^{r}(f^{\circ i}(X))^{n_i}$$
Notice that $f_n(X)=X^n$ for $n=0,\ldots,q-1$ and $f_q(X)=f(X)$. Moreover, $f_n\in\Int_{\Q_p}(\mathcal O_{K})$ and has degree $n$, for every $n\in\N$. It is easy to prove by induction that ${\rm lc}(f^{\circ i})=p^{-a_i}$, where $a_i=1+q+\ldots+q^{i-1}=w_q(q^i!)$. By the same proof of  \cite[Chap. 2, Prop. II.2.12]{CaCh} one can show that ${\rm lc}(f_n)=p^{-w_q(n)}$
for every $n\in\N$, so, finally, the family of polynomials $\{f_n(X)\}_{n\in\N}$ is a regular basis of $\Int_{\Q_p}(\mathcal{O}_K)$.
\end{remark}

\begin{example}\label{remark3.5} 
In the next two examples we show the assumptions in Theorem \ref{equality-theorem} cannot be dropped. 
 
(1) 
If $K/\Q_p$ is not a Galois extension, then the restriction of the trace homomorphism to $\Int(\mathcal{O}_K)$ may give a polynomial in $\Q_p(X)$ which is not in $\Int_{\Q_p}(\mathcal{O}_K)$. 
For example,  let $K=\Q_2(\sqrt[3]{2})$, whose ring of integers is $\mathcal O_K=\Z_2[\sqrt[3]{2}]$. Then the polynomial 
\[f(X)=\frac{X(X-1)(X-\sqrt[3]{2})(X-(1+\sqrt[3]{2}))}{2}\] 
%(obtained in \cite{CaCh}, or of the notion of $P$-ordering by Bhargava \cite{Bha}) 
is in $\Int(\mathcal O_K)$ but its trace over $\Q_2(X)$ is equal to $g(X)=\frac{3X^2(X-1)^2}{2}$, which is not integer-valued over $\mathcal O_K$, since $g(\sqrt[3]{2})\notin \mathcal{O}_K$. 
One can show by an explicit computation that in this example the equality 
$w_{\mathcal O_K}^{\Q_p}(n)=\left\lfloor \frac{w_q(n)}{e}\right\rfloor$ does not hold for $n=4$. Indeed, the first four elements of a 
$\mathcal O_K$-basis of $\Int(\mathcal O_K)$ are 
\[f_1(X)=X; \quad f_2(X)=\frac{X(X-1)}{\sqrt[3]2};  \quad 
f_3(X)=\frac{ X(X-1)(X-\sqrt[3]{2})}{\sqrt[3]{2}}; \]
\[f_4(X)=\frac{X(X-1)(X-\sqrt[3]{2})(X-(1+\sqrt[3]{2}))}{2};\]
and considering all possible $\mathcal O_K$-combinations of these elements which lie in $\Q_2[X]$  (recall that $\Int_{\Q_2}(\mathcal{O}_K)=\Q_2[X]\cap\Int(\mathcal{O}_K)$) ,   
we see that 
 there is no element in $\Int_{\Q_2}(\mathcal O_K)$ 
of degree $4$ whose leading coefficient has valuation $-1=-\left\lfloor\frac{w_2(4)}{3}\right\rfloor$.

(2) 
We now discuss the tameness assumption. 
Consider the case of $K=\Q_2(i)$ with $i^2=-1$ and let $\{f_n(X):n\geq 0\}$ be  
a regular basis of $\Int(\mathcal O_K)$ obtained by means of compositions and products of the Fermat polynomial $\frac{X^2-X}{1+i}$ (in the same way as in the Example \ref{unramified explicit bases}; see \cite[Chapter II, p. 32]{CaCh}).
We set $G(X)=X^2-X$. 
One can check that 
\[f_6+if_4= -\frac{G^3}{4}+\frac{G^2}{2}-\frac{G}{2}\] and 
\[f_{10}+2f_8-2if_6+(1-2i)f_4=\frac{G^5}{16}+\frac{G^3}{8}-\frac{G^2}{4}+G\] 
belong to $\Int_{\Q_2}(\mathcal O_K)$ and their leading coefficients have 
valuation equal to $-\left\lfloor \frac{w_2(6)}{2}\right\rfloor=-2$ and 
$-\left\lfloor \frac{w_2(10)}{2}\right\rfloor=-4$, 
respectively; one can also check that 
\[-v_2\left(\mathfrak I_n\left(\Int_{\Q_2}(\mathcal O_K)\right)\right)=\left\lfloor \frac{w_2(n)}{2}\right\rfloor\]
for all $n\leq 11$. On the other hand, writing down a basis of $\Int(\mathcal O_K)$ up to degree $12$, and 
considering all possible $\mathcal O_K$-combinations of these elements which lie in $\Q_2[X]$,   
we see that 
\[-v_2\left(\mathfrak I_{12}\left(\Int_{\Q_2}(\mathcal O_K)\right)\right)=\left\lfloor \frac{w_2(12)}{2}\right\rfloor-1.\]
It might be interesting to describe the values taken by 
$v_p\left(\mathfrak I_{n}\left(\Int_{\Q_p}(\mathcal O_K)\right)\right)$ in the case of wild ramification. 
\end{example}

\subsection{Global case} 
Let $K/\Q$ be a finite Galois extension with absolute discriminant $D$ and degree $d$ over $\Q$.  For each rational prime $p$, denote by $f_p$ the residue class degree and $e_p$ its ramification degree in $O_K$.  As usual, we say that $K/\Q$ is tamely ramified if, for every prime $p\in\Z$, $p\nmid e_p$ . Let $q_p= p^{f_p}$ be the cardinality of the residue field of $K_p$. The following is a reformulation of Theorem \ref{Thm-Intro2} in the Introduction:

\begin{theorem}\label{prop3.5}
 Let $K/\Q$ be a tamely ramified Galois extension. Then 
\[\mathfrak I_n(\Int_\Q(\mathcal O_{K}))= \left(\prod_pp^{-\left\lfloor \frac{w_{q_p}(n)}{e_p}\right\rfloor}\right)\]
as fractional ideals of $\Z$, 
where the product is over the set of all primes $p\in\Z$. 
\end{theorem} 
\begin{proof} Note that for a fixed $n$ we have  $w_q(n)=0$  for almost all prime powers $q$, and therefore 
the above product is well defined. 
The result follows immediately combining Proposition \ref{prop2.1} and Theorem \ref{equality-theorem}.\qed
\end{proof}

\subsection*{\textbf{Acknowledgments.}}

The authors wish to thank the referee for carefully reading the paper.

The first author is partially supported by PRAT 2013 ``Arithmetic of Varieties over Number Fields". 
The second author is partially supported by PRIN 2010/11 ``Arithmetic Algebraic Geometry and Number Theory" 
and PRAT 2013 ``Arithmetic of Varieties over Number Fields".
The third author has been supported by grant ``Bando Giovani Studiosi 2013'', Project title  ``Integer-valued polynomials over algebras'' Prot. GRIC13X60S of the University of Padova. 

%\noindent 

\end{document}